\documentclass{amsart}

\usepackage[margin=3.5cm]{geometry}

\usepackage{amsmath,amscd,amssymb,amsfonts,latexsym,wasysym, mathrsfs, mathtools,hhline,color}
\usepackage[all, cmtip]{xy}

\usepackage{csquotes}
\usepackage{url}

\definecolor{hot}{RGB}{65,105,225}

\usepackage[pagebackref=true,colorlinks=true, linkcolor=hot ,  citecolor=hot, urlcolor=hot]{hyperref}

\usepackage{ textcomp }
\usepackage{ tipa }

\usepackage{graphicx,enumerate}

\theoremstyle{plain}
\newtheorem{theorem}{Theorem}[section]
\newtheorem{prop}[theorem]{Proposition}

\newtheorem{lm}[theorem]{Lemma}

\newtheorem{cor}[theorem]{Corollary}
\newtheorem{conj}[theorem]{Conjecture}

\newtheorem{thrm}[theorem]{Theorem}

\theoremstyle{definition}

\newtheorem{defn}[theorem]{Definition}

\newtheorem{rmk}[theorem]{Remark}

\newtheorem{ex}[theorem]{Example}
\newtheorem*{ex*}{Example}

\def\be{\begin{equation}}
\def\ee{\end{equation}}

\def\bt{\begin{thrm}}
\def\et{\end{thrm}}

\def\bc{\begin{cor}}
\def\ec{\end{cor}}

\def\br{\begin{rmk}}
\def\er{\end{rmk}}

\def\bp{\begin{prop}}
\def\ep{\end{prop}}

\def\bl{\begin{lm}}
\def\el{\end{lm}}

\def\bex{\begin{ex}}
\def\eex{\end{ex}}

\def\bd{\begin{defn}}
\def\ed{\end{defn}}

\newcommand\sF{{\mathcal F}}

\DeclareMathOperator{\reg}{reg}                  
                  
\DeclareMathOperator{\sFc}{\sF^{\centerdot}}

\def\bC{\mathbb{C}}

\def\lra{\longrightarrow}

\def\bZ{\mathbb{Z}}

\newcommand{\C}{\mathbb{C}}

\def\part{\partial}

\def\lra{\longrightarrow}

\def\bd{\begin{defn}}
\def\ed{\end{defn}}
\def\bt{\begin{thm}}
\def\et{\end{thm}}
\def\br{\begin{remark}}
\def\er{\end{remark}}
\def\bc{\begin{cor}}
\def\ec{\end{cor}}
\def\bp{\begin{prop}}
\def\ep{\end{prop}}
\def\be{\begin{equation}}
\def\ee{\end{equation}}
\def\bn{\begin{enumerate}}
\def\en{\end{enumerate}}
\def\ba{\begin{array}}
\def\ea{\end{array}}
\def\bex{\begin{example}}
\def\eex{\end{example}}

\def\bC{\mathbb{C}}

\def\lra{\longrightarrow}

\def\bZ{\mathbb{Z}}
\def\bZ{\mathbb{Z}}

\title[Singular Singer-Hopf Conjecture]{On singular variants of the Singer-Hopf Conjecture}

\author{Lauren\c tiu Maxim}
\address{Department of Mathematics,         University of Wisconsin-Madison,  480 Lincoln Drive, Madison WI 53706-1388, USA.}
\email {maxim@math.wisc.edu}

\date{\today}

\keywords{Euler characteristic, aspherical manifold, Singer-Hopf conjecture, constructible function, characteristic cycle, nef vector bundle}

\dedicatory{Dedicated to the memory of Prof. Jean-Pierre Demailly}

\subjclass[2010]{14F35, 14F45, 14C17, 32S60, 58K30}

\begin{document}

\maketitle
\begin{abstract}  
We propose singular variants of the Singer-Hopf conjecture, formulated in terms of the Euler-Mather characteristic, intersection homology Euler characteristic and, resp., virtual Euler characteristic of a closed irreducible subvariety of an aspherical complex projective manifold. We prove the conjecture under the assumption that the cotangent bundle of the ambient variety is numerically effective (nef), or, more generally, when the ambient manifold admits a finite morphism to a complex projective manifold with a nef cotangent bundle. 
\end{abstract}



\section{Introduction}\label{intro}

The main purpose of this note is to overview and enhance some of the recent developments \cite{AW, LMW} around the Singer-Hopf conjecture in the complex algebraic context. In order to reach a wider audience, results are formulated in the convenient language of constructible functions. Along the way, we generalize a result of \cite{AMSS} on the non-negativity of Euler characteristics of an important class of constructible functions.

Recall that a connected CW complex is {\it aspherical} if its universal cover is contractible. Closed Riemannian manifolds with non-positive sectional curvature are aspherical. The following conjecture on the Euler characteristic of a closed aspherical manifold  was made by Hopf (and later on strengthened by Singer):
\begin{conj}[Singer-Hopf]\label{SH} If $X$ is a closed aspherical manifold of real dimension $2n$, then \be\label{ho}(-1)^n  \cdot \chi(X) \geq 0.\ee \end{conj}
The conjecture is clearly true for $n=1$ (i.e., real dimension $2$). 
For closed Riemannian manifolds with non-positive sectional curvature, the conjecture is due to Hopf and Chern, and it was mentioned in the list of problems \cite{Yau} compiled by Yau in 1982; it it true in this case for $n=2$ since the Gauss-Bonnet integrand has the desired sign, cf. \cite[Theorem 5]{Ch} where the proof is attributed to Milnor. To our knowledge, Conjecture \ref{SH} is not known for all closed aspherical $4$-manifolds, and it is open for $n \geq 3$. 

In \cite{Gro}, Gromov introduced the notion of K\"ahler hyperbolicity, including compact K\"ahler manifolds with negative sectional curvature, and he verified Conjecture \ref{SH} for K\"ahler hyperbolic manifolds. Cao-Xavier \cite{CX} and  Jost-Zuo \cite{JZ} independently introduced the concept of K\"ahler nonellipticity, including  compact K\"ahler manifolds with non-positive sectional curvature, and settled the Singer-Hopf conjecture in this case (cf. also \cite{Es}). All these works proved a corresponding version of Conjecture \ref{SH} by means of vanishing $L^2$-cohomology. More recently, Liu-Maxim-Wang \cite{LMW} proved the complex projective version of Conjecture \ref{SH} under the additional assumption that the (holomorphic) cotangent bundle $T^*X$ of $X$ is {\it numerically effective} ({\it nef}, for short), and they conjectured that aspherical complex projective manifolds have nef cotangent bundles.
Finally, Arapura-Wang gave in \cite{AW} a new proof of Conjecture \ref{SH} for compact K\"ahler manifolds  with non-positive sectional curvature, using the fact that the cotangent bundle of such a manifold is nef.

In this note we propose the following {\it singular} variants of the Singer-Hopf Conjecture \ref{SH} in the complex projective context:

\begin{conj}\label{co10} If $Z \subseteq X$ is a closed irreducible subvariety of an aspherical complex projective manifold $X$, then:
\begin{itemize}
\item[(i)] $(-1)^{\dim_\C Z} \cdot \chi(Z, Eu_Z) \geq 0$, where $Eu_Z$ is the {local Euler obstruction} function of MacPherson and $\chi(Z, Eu_Z)$ is the corresponding Euler-Mather characteristic.  
\item[(ii)] $(-1)^{\dim_\C Z} \cdot \chi^{IH}(Z) =\chi(Z, ic_Z)\geq 0$, where $\chi^{IH}(Z)$ denotes the intersection cohomology Euler characteristic of $Z$, with $ic_Z$ defined by taking the stalkwise Euler characteristic of the $IC$-complex $IC_Z$ on $Z$.
\item[(iii)] $\chi_{vir}(Z):=\chi(Z, \nu_Z) \geq 0$, where $\nu_Z$ is Behrend's constructible function of $Z$ and $\chi_{vir}(Z)$ is the corresponding Donaldson-Thomas invariant. 
\end{itemize}
\end{conj}

If $Z=X$, then all statements in Conjecture \ref{co10} reduce to the complex projective version of the Singer-Hopf Conjecture \ref{SH}. More generally, if $Z$ a smooth closed irreducible subvariety of the aspherical complex projective manifold $X$, all statements in Conjecture \ref{co10} become $(-1)^{\dim_\C Z} \cdot \chi(Z) \geq 0$.

In fact, we make the following more uniform conjecture (which, as explained in Proposition \ref{p21}, turns out to be equivalent to Conjecture \ref{co10}(i)).
\begin{conj}\label{co2} 
Let $X$ be an aspherical complex projective manifold and let $\varphi$ be a constructible function on $X$ with effective characteristic cycle.
Then the Euler characteristic of $\varphi$  is non-negative, that is, $\chi(X,\varphi) \geq 0$.
\end{conj}

If $X$ is an abelian variety, Conjecture \ref{co2} can be deduced from \cite[Theorem 1.3]{FK}, see also \cite{AMSS} and \cite{EGM}. When $\varphi$ is obtained (by taking the stalkwise Euler characteristic) from a perverse sheaf, Conjecture \ref{co2} reduces to \cite[Conjecture 6.2]{LMW}.

\smallskip

In this note, we prove Conjecture \ref{co2} under the additional assumption that the cotangent bundle $T^*X$ of $X$ is nef (e.g., globally generated), or, more generally, if $X$ admits a finite morphism to a complex projective manifold with nef cotangent bundle. Moreover, the inequalities become strict if ``nef'' is replaced by ``ample''. 

\begin{theorem}\label{th-main}
Let $X$ be a complex projective manifold and let $\varphi$ be a constructible function on $X$ with effective characteristic cycle. Assume $X$ admits a finite morphism $f:X \to Y$ to a complex projective manifold $Y$ with nef cotangent bundle (e.g., $Y$ has non-positive sectional curvature). Then $\chi(X,\varphi) \geq 0$. Moreover, the inequality is strict if $T^*Y$ is ample (e.g., $Y$ has negative sectional curvature). 
\end{theorem}

A weaker version of Conjecture \ref{co2}, for $\varphi$ coming from a perverse sheaf on $X$ and $T^*X$ nef,  was  proved in \cite[Proposition 3.6]{LMW}. While the proof of Theorem \ref{th-main} follows the same lines as that of loc.cit. (see also Theorem \ref{mapp} in Section \ref{positive} for a more general statement), what we want to emphasize here are the various facets of Conjecture \ref{co2}, reflected in the statement of Conjecture \ref{co10},
if more general coefficients are allowed. Note that asking for the characteristic cycle of a constructible function to be effective is much weaker that asking for that function to come from a perverse sheaf. 
Conjecture \ref{co10} is obtained directly from Conjecture \ref{co2} by letting $\varphi$ be one of the following constructible functions supported on $Z$: $(-1)^{\dim_\C Z} Eu_Z$, $ic_Z$ and $\nu_Z$, respectively, all of which are known to have effective characteristic cycles, e.g., see the discussion in \cite[Section 7]{AMSS}. 
Note that, in the notations of Section \ref{prelim}, only $ic_Z$ comes from a perverse sheaf.

We prove Theorem \ref{th-main} (see also Theorem \ref{mapp}) following the same approach as in \cite{LMW} (see also \cite{AW}), based on Kashiwara's index theorem for constructible complexes of sheaves (cf. \cite{Kas}), together with (semi-)positivity results for nef (resp., ample) bundles (cf. \cite{DPS, FuLa}). It is well-known that if $X$ has non-positive (resp., negative) sectional curvature, then $T^*X$ is nef (resp., ample); see, e.g., \cite{DPS} or \cite[Lemmas 4.1, 4.2]{AW}. And, as already mentioned, it was conjectured in \cite{LMW} that aspherical complex projective manifolds have nef cotangent bundles.

\smallskip

It was also proved in \cite{AW} that Conjecture \ref{co2} is true if $X$ is an aspherical complex projective manifold (or, more generally, if $X$ has a large fundamental group \cite{Ko})  which  admits a cohomologically rigid almost faithful semi-simple representation, provided that $\varphi$ comes from a perverse sheaf on $X$. 
In Section \ref{positive}, we note that the proof of this result in loc.cit extends to all constructible functions with effective characteristic cycles, since the only operations involved in the proof preserve the effectivity of characteristic cycles of constructible functions.

\medskip

The paper is organized as follows. In Section \ref{prelim}, we review the relevant background about constructible complexes, constructible functions, characteristic cycles, and prove (in Proposition \ref{p21}) that Conjecture \ref{co2} is equivalent to Conjecture \ref{co10}(i). In Section \ref{positive}, we review (semi-)positivity results for nef and, resp., ample vector bundles on complex projective manifolds, and prove Theorems \ref{th-main} (as a consequence of the more general statement of Theorem \ref{mapp}).  
We conclude with a discussion around another conjecture of Hopf, whose proof in the projective/K\"ahler context is known to follow from classical results in complex algebraic geometry.

\medskip

{\bf Acknowledgments.} The author thanks Botong Wang and J\"org Sch\"urmann for useful discussions. Partial support was provided by the Simons Foundation and by  the  Romanian  Ministry  of  National  Education.


\section{Preliminaries: constructible complexes and characteristic cycles}\label{prelim}

Let $X$ be a complex algebraic manifold. We denote by $D^b_c(X)$ the bounded derived category of $\C$-constructible complexes  on $X$. Consider the functor $$CC:K_0(D^b_c(X)) \lra LCZ(T^*X)$$
which associates characteristic cycles to (Grothendieck classes of) constructible complexes on $X$ (e.g., see \cite[Definition 4.3.19]{Di} or \cite[Chapter IX]{KS}). Here, we let $LCZ(T^*X)$ denote the free abelian group spanned by the irreducible conic Lagrangian cycles in the cotangent bundle $T^*X$. Its elements are of the form $\sum_Z n_Z \cdot T^*_{Z}X$, for some $n_Z \in \bZ$ and $Z$ closed irreducible subvarieties of $X$. Recall that, if $Z$ is a closed irreducible subvariety of $X$ with smooth locus $Z_{\reg}$, its conormal bundle $T^*_{Z}X$ is  the closure in $T^*X$ of $$T^*_{Z_{\reg}}X:=\{
(z,\xi)\in T^*X \mid z \in Z_{\reg}, \ \xi \in T^*_zX, \ \xi\vert_{T_zZ_{\rm reg}}=0 \}.$$
One then has a group isomorphism
$$T:LCZ(T^*X) \lra Z(X)$$ to the group $Z(X)$ of algebraic cycles on $X$, defined on generators by: $T^*_ZX \mapsto (-1)^{\dim_\C Z} Z.$

A function $\varphi:Z \to \bZ$ on a complex algebraic variety (e.g., a closed subvariety of $X$) is said to be {\it constructible} if there is a Whitney stratification $\mathcal{S}$ of $Z$ so that $\varphi$ is constant on each (connected) stratum $S \in \mathcal{S}$. The Euler characteristic of such a constructible function $\varphi$ is defined by
$$\chi(Z,\varphi):=\sum_{S \in \mathcal{S}} \chi(S) \cdot \varphi_{\vert_S}.$$
 Let $F(Z)$ be the group of constructible functions on  $Z$. 
 
To any bounded constructible complex $\sF^{\centerdot} \in D^b_c(Z)$ on a complex algebraic variety $Z$, one associates a constructible function $\chi_{st}(\sF^{\centerdot})\in F(Z)$ by taking the stalkwise Euler characteristic, i.e.,
$$\chi_{st}(\sF^{\centerdot})(z):=\chi(\sF^{\centerdot}_z)$$
for any $z \in Z$. For example, $\chi_{st}(\bC_Z)=1_Z$, the indicator function of $Z$, and if $Z$ is pure-dimensional we let $$ic_Z:=\chi_{st}(IC_Z),$$ where $IC_Z$ is the intersection cohomology ($IC$-)complex of $Z$. Note that if $\varphi=\chi_{st}(\sFc)$, then $$\chi(Z,\varphi)=\chi(Z,\sFc).$$

For a closed subvariety  $Z$ in $X$, a constructible function (or complex) on $Z$ can be regarded as a constructible function (or complex) on $X$ by extending it by $0$ on $X \setminus Z$, hence for $\varphi\in F(Z)$ we have $\chi(Z,\varphi)=\chi(X,\varphi)$. For the purpose of this note we may, without any loss of generality, work with constructible functions on $X$ (with support in a closed subvariety) and their corresponding Euler characteristics.

The Euler characteristic induces by additivity an epimorphism
$$\chi_{st}:K_0(D^b_c(X)) \lra F(X).$$
Moreover, since the class map $D^b_c(X) \to K_0(D^b_c(X))$ is onto, $\chi_{st}$ is already an epimorphism on $D^b_c(X)$. The usual functors in sheaf theory, which respect the corresponding category of bounded constructible complexes, induce via $\chi_{st}$ well-defined group homomorphisms on the level of constructible functions (see, e.g., \cite[Section 2.3]{S}).

Another important example of a constructible function  on a complex algebraic variety $Z$ is the MacPherson {\it local Euler obstruction} function $Eu_Z$, see \cite{MP0}. When $Z$ is a closed irreducible subvariety $Z$ of $X$,  $Eu_Z$ can be seen as a function defined on  all of $X$ by setting $Eu_Z(x)=0$ for $x \in X \setminus Z$. In particular, one may consider the group homomorphism
\begin{equation}\label{eq_ZCF}
Eu:Z(X) \lra F(X)
\end{equation}
defined on an irreducible cycle $Z$ by the assignment $Z \mapsto Eu_Z$, and then extended by $\bZ$-linearity. It is well known (e.g., see \cite[Theorem 4.1.38]{Di} and the references therein), that the homomorphism
$Eu:Z(X) \to F(X)$
is an isomorphism.

The local Euler obstruction function appears in the formulation of the {\it local index theorem}, which in the above notations asserts the existence of the following commutative diagram  (e.g., see \cite[Section 5.0.3]{S} and the references therein):
\begin{equation}\label{eq_xy}
\xymatrix{
K_0(D^b_c(X)) \ar[d]_{CC} \ar[r]^{\chi_{st}} & F(X)   \\
LCZ(T^*X) \ar[r]_T^{\cong} & Z(X) \ar[u]_{Eu}^{\cong}
}
\end{equation}
In particular, one can associate a characteristic cycle to any constructible function. 
For example, if $Z$ is a closed irreducible subvariety of $X$, one has:
\be\label{cc} CC(Eu_Z)=(-1)^{\dim_\C Z}\cdot T^*_Z X.\ee
Note also that $$CC(\sF^{\centerdot})=CC(\chi_{st}(\sF^{\centerdot}))$$ for any constructible complex $\sF^{\centerdot} \in D^b_c(X)$.

Kashiwara's {\it global index theorem} \cite{Kas} computes the Euler characteristic of any bounded constructible complex $\sFc$ on $X$ with $supp(\sFc)$ compact, or, equivalently, that of the constructible function $\varphi=\chi_{st}(\sFc) \in F(X)$, by the formula:
\be\label{git} \chi(X,\sFc)=\chi(X,\varphi)= CC(\varphi) \cdot T^*_XX, \ee
that is, the intersection index  in the complex manifold $T^*X$, of the {characteristic cycle} of $\varphi$ with the zero section of $T^*X$.

\bd
If $\varphi\neq 0 \in F(X)$ with $CC(\varphi)=\sum_Z n_Z \cdot T^*_{Z}X$, we say that $CC(\varphi)$ is {\it effective} if all coefficients $n_Z$ are positive. 
\ed

For instance, \eqref{cc} implies that $CC\left((-1)^{\dim_\C Z} Eu_Z\right)$ is effective. It is also well known that if $0\neq \varphi=\chi_{st}(\sFc)$ is a nontrivial constructible function associated to a peverse sheaf $\sFc$ on $X$, then $CC(\varphi)$ is effective (e.g., see \cite[Corollary 4.7]{MS}). In particular, $CC(ic_Z)$ is effective. Finally, the Behrend function $\nu_Z$ has an effective characteristic cycle (see \cite{AMSS,Beh}).

\medskip

We can now prove the following.

\bp\label{p21}
Conjecture \ref{co10}(i) is equivalent to Conjecture \ref{co2}.
\ep

\begin{proof}
Conjecture \ref{co2} implies Conjecture \ref{co10}(i) by choosing $\varphi=(-1)^{\dim_\C Z} Eu_Z$ (extended by $0$ to all of $X$).

Conversely, assume Conjecture \ref{co10}(i) holds for all closed irreducible subvarieties $Z$ of $X$. Let $\varphi \in F(X)$ with $CC(\varphi)$ effective. Then
\be CC(\varphi)=\sum_{Z \subseteq X} n_Z \cdot T^*_Z X\ee
for uniquely determined closed irreducible subvarieties $Z$ of $X$ and positive integers $n_Z$. By \eqref{cc}, the coefficients $n_Z$ are determined by the following equality of constructible functions
\be 0 \neq \varphi =\sum_{Z \subseteq X} n_Z \cdot (-1)^{\dim_\C Z} \cdot Eu_Z.\ee
Hence $$\chi(X,\varphi)=\sum_{Z \subseteq X} n_Z \cdot  (-1)^{\dim_\C Z} \cdot \chi(Z, Eu_Z) \geq 0,$$
with the last inequality following by applying Conjecture \ref{co10}(i) to each subvariety $Z$ in the support of $CC(\varphi)$.
\end{proof}

{
In \cite[Proposition 7.2]{AMSS}, the authors list some basic operations of constructible functions which preserve the property of having an effective characteristic cycle. In particular, one has the following. 
\bp\cite[Proposition 7.2(2)]{AMSS}\label{ppf}
Let $Z$ be a closed reduced subscheme of a smooth complex algebraic variety $X$, and assume that $\varphi$ is a constructible function on $Z$ with $CC(\varphi)$ effective. 
\begin{itemize}
\item[(i)] Let $f:Z \to Z'$ be a finite morphism, with $Z'$ a closed reduced subscheme of a smooth complex algebraic variety $X'$. Then $f_*(\varphi)$ is a constructible function on $Z'$ with $CC(f_*(\varphi))$ effective. Here, $f_*(\varphi)$ is the constructible function on $Z'$ defined by:
$$f_*(\varphi)(z'):=\sum_{z \in f^{-1}(z')} \varphi(z).$$
\item[(ii)] Let $f:X' \to X$ be a morphism of smooth complex algebraic varieties such that $f\colon Z':=f^{-1}(Z) \to Z$ is a smooth morphism of relative dimension $d$. Then $(-1)^d f^*(\varphi)=(-1)^d \varphi \circ f$ is a constructible function on $Z'$ with effective characteristic cycle.
\end{itemize}
\ep
}

As already mentioned, the above functors of constructible functions coincide with those induces via $\chi_{st}$ from the corresponding functors of constructible complexes of sheaves.

\section{Positivity results for nef bundles. Applications}\label{positive}

 In this section, we recall (semi-)positivity results for ample (resp., nef) vector bundles on complex projective manifolds. We use such results to deduce (semi-)positivity statements for the Euler characteristics of constructible functions with effective characteristic cycles, thus proving Theorem \ref{th-main}. 
 
\bd\rm
If $\mathcal{E}$ is a vector bundle on a  complex projective manifold $X$, denote by $\mathbf{P}(\mathcal{E})$ the projective bundle of hyperplanes in the fibers of $\mathcal{E}$. 
A vector bundle $\mathcal{E}$ on $X$ is called {\it ample} (resp. {\it nef}) if the line bundle $\mathcal{O}(1)$ on $\mathbf{P}(\mathcal{E})$ is ample (resp. nef) in the classical sense. 
\ed

The nef condition is a degenerate ampleness condition. Properties of nef/ample bundles are studied, e.g.,  in \cite{DPS, FuLa, La}.

The following semi-positivity result was proved by Fulton-Lazarsfeld \cite{FuLa} for ample bundles, and extended to nef bundles by Demailly-Peternell-Schneider \cite{DPS} (cf., e.g., \cite[Section 8.1.B]{La} for the definition of the intersection number).
\begin{theorem}{\rm (\cite[Proposition 2.3]{DPS}, \cite[Theorem II]{FuLa}))}\label{thm_nef}
Let $X$ be a complex projective manifold and let $E$ be a rank $r$ nef (resp., ample) vector bundle on $X$. For any $r$-dimensional conic subvariety $C$ of $E$, one has
\[
 C \cdot  Z_E  \geq 0 \ \ \ ({\rm resp.}, \ >0),
\]
where $Z_E$ is the zero section of $E$, and $C \cdot  Z_E$ denotes the intersection number of cycles in $E$. 
\end{theorem}

\begin{rmk}
The notion of nef vector bundle can be extended to the K\"ahler manifold context, with nefness of a line bundle understood in the sense of \cite[Definition 1.2]{DPS}; this coincides with the usual definition in the projective case. Then the above theorem holds more generally, for nef vector bundles on compact K\"ahler manifolds (cf. \cite[Proposition 2.3]{DPS}).
\end{rmk}

The first result of this section involves nef/ample cotangent bundles. As already mentioned in the Introduction, complex projective manifolds with non-positive (resp., negative) sectional curvature have nef (resp., ample) cotangent bundle. 

\begin{ex}
The class of complex projective manifolds whose cotangent bundles are nef is closed under taking finite (unramified) covers, products, and subvarieties, and it includes smooth subvarieties of abelian varieties.  However, if $A$ is an abelian variety of dimension $n$ and $X$ is a smooth subvariety of $A$ of dimension $d$ and codimension $n-d<d$, then the cotangent bundle of $X$ is not ample (e.g., see \cite[Example 7.2.3]{La}). On the other hand, for an arbitrary $m$-dimensional complex projective manifold $M$ and each $n \leq m/2$, there are plenty of smooth $n$-dimensional subvarieties $X$ of $M$ with ample cotangent bundle (e.g., complete intersections of sections of $M$ by general hypersurfaces of sufficiently high degrees in the ambient projective space, see \cite{BD, X}). Finally, Kratz \cite[Theorem 2]{Kr} showed that if $X$ is a complex projective manifold whose universal cover is a bounded domain in $\C^n$ or in a Stein manifold, then $T^*X$ is nef. This result prompted Liu-Maxim-Wang \cite{LMW} to conjecture that the nefness of the cotangent bundle should hold more generally, if the universal cover is Stein, a claim refuted recently in \cite{W}.
\end{ex}

We can now prove the following result.
\begin{theorem}\label{mapp}
Let $X$ be a complex algebraic variety, and let $f\colon X \to Y$ be a morphism to a complex projective manifold $Y$ with $T^*Y$ nef. Let $\varphi \in F(X)$ be a constructible function on $X$ such that the characteristic cycle of the constructible function $f_*(\varphi) \in F(Y)$ is effective. Then $\chi(X,\varphi) \geq 0$, and the inequality is strict if $T^*Y$ is ample.
\end{theorem}
\begin{proof}
First note that \begin{equation}\label{eq1}\chi(X,\varphi)=\chi(Y, f_*(\varphi)),\end{equation}
so it suffices to show that $\chi(Y, f_*(\varphi)) \geq 0$, with strict inequality if $T^*Y$ is ample. Let
$$CC(f_*(\varphi))=\sum_{Z\subseteq Y} n_Z  \cdot T^*_ZY,$$ for uniquely determined  closed irreducible subvarieties $Z$ of $Y$, and $n_Z >0$ by the effectivity assumption. Since $T^*Y$ is nef, one gets by Theorem \ref{thm_nef} that $$T^*_ZY \cdot T^*_YY \geq 0,$$ with strict inequality if $T^*Y$ is ample.  Kashiwara's global index formula \eqref{git} then yields:
\begin{equation}\label{eq2}\chi(Y,f_*(\varphi))=\sum_{Z \subseteq Y} n_Z  \left( T^*_ZY \cdot T^*_YY\right) \geq 0,\end{equation}
with strict inequality if $T^*Y$ is ample.
The assertion of the theorem follows now by combining \eqref{eq1} and \eqref{eq2}.
\end{proof}

Theorem \ref{th-main} is an immediate consequence of Theorem \ref{mapp}, as we now show.
\begin{proof}[Proof of Theorem \ref{th-main}]
Let $f:X \to Y$ be a finite morphism between complex projective manifolds,  with $T^*Y$ nef. 
Since $\varphi \in F(X)$ has, by assumption, an effective characteristic cycle, it follows from Proposition \ref{ppf}(i) that  $CC(f_*(\varphi))$ is effective in $T^*Y$. The assertion follows now from Theorem \ref{mapp}.
\end{proof}

\begin{ex}
Besides the situation considered in Theorem \ref{th-main}, there are other interesting classes of morphisms $f:X \to Y$ (and constructible functions $\varphi \in F(X)$ defined on their domain) satisfying the assumptions of Theorem \ref{mapp}. For instance, if $\varphi \in F(X)$ comes from a perverse sheaf, let $f:X \to Y$ be a morphism such that $Rf_*$ preserves perverse sheaves (e.g., a closed embedding or a quasi-finite affine morphism); then $f_*(\varphi)\in F(Y)$ also comes from a perverse sheaf and hence, since $Y$ is assumed smooth, $CC(f_*(\varphi))$ is effective in $T^*Y$. Another example is provided by the projection $f:X \to Y$ from a semi-abelian variety $X$ onto its abelian part $Y$. In this case, it was shown in \cite[Section 8]{AMSS} that if $\varphi \in F(X)$ comes from a perverse sheaf, then $f_*(\varphi)\in F(Y)$ also comes from a perverse sheaf. Hence  $CC(f_*(\varphi))$ is effective, and Theorem \ref{mapp} applies to give $\chi(X,\varphi) \geq 0$, an inequality initially proved in \cite[Corollary 1.4]{FK}. For more examples in this direction, the interested reader may consult \cite[Proposition 8.4, Example 8.5]{AMSS}.
\end{ex}

We next discuss the following extension of the main result of \cite{AW} to the case of effective characteristic cycle.
\begin{theorem}\label{34}
Let $X$ be a smooth complex projective variety with large algebraic fundamental group (e.g., $X$ is aspherical). Suppose that there exists a cohomologically rigid almost faithful semi-simple representation $\rho:\pi_1(X) \to GL(r,\bC)$. Let $\varphi$ be a constructible function on $X$ with $CC(\varphi)$ effective. Then $\chi(X,\varphi) \geq 0$.
\end{theorem}
\begin{proof}
The proof of this result  in \cite{AW} for perverse sheaves coefficients can be readily extended to the more general setup of effective characteristic cycles, since the only operations involved in the proof preserve the effectivity of characteristic cycles of constructible functions. For the benefit of the reader, we summarize the main steps in the proof, by adapting \cite[Theorems 1.8 and 1.6]{AW} to the setup of our paper.

Let $L_\rho$ be the local system on $X$ corresponding to the representation $\rho$. As noted in \cite{AW}, $L_\rho$ underlies a complex variation of Hodge structure (CVHS) $V$ with discrete monodromy. 

{\it Step 1.} \ Upon passing to a finite unramified cover $\pi:X'\to X$, 
one can moreover assume that the monodromy group $\Gamma=\rho(\pi_1(X))$ of $L_\rho$ is torsion free (cf. \cite[Lemma 5.1]{AW}). After performing this step, the Euler characteristic $\chi(X,\varphi)$ gets multiplied by the degree of the cover $\pi$ and hence it preserves its sign, and the characteristic cycle of $\pi^*(\varphi)$ is still effective (cf. Proposition \ref{ppf}(ii)).

The effect of performing Step 1 is that the quotient $M=\Gamma \backslash D$ of the Griffiths period domain by the monodromy group $\Gamma$ is now a manifold, and the period map of the CVHS $V$ induces a horizontal map $\alpha:X \to M$ (i.e., its derivative lies in the subbundle $T^hM \subset TM$ induced from the horizontal subbundle of $D$).

{\it Step 2.} \ Since $\chi(X,\varphi)=\chi(M,\alpha_*(\varphi))$, it thus suffices to show that $\chi(M,\alpha_*(\varphi)) \geq 0$.
Conditions of the theorem force $\alpha :X \to M$ to be quasi-finite, and hence a finite morphism (since $X$ is compact).  Hence, by Proposition \ref{ppf}(i), the characteristic cycle of the constructible function $\alpha_*(\varphi)$ (supported on the compact subvariety $\alpha(X)$) is effective in $T^*M$. So if $$CC(\alpha_*(\varphi))=\sum_Z n_Z \cdot T^*_ZM,$$
it suffices to show that for any subvariety $Z$ appearing in the sum one has \be\label{71} T^*_ZM \cdot T^*_MM \geq 0.\ee Griffiths transversality implies that $\alpha(X)$  is a horizontal subvariety of $M$ (in the sense of \cite[Definition 5.2]{AW}), and hence by \cite[Lemma 5.3]{AW}, so is any irreducible subvariety appearing in $CC(\alpha_*(\varphi))$. The desired inequality \eqref{71} follows now from \cite[Proposition 6.1]{AW}.
\end{proof}


We conclude this note with the following discussion about complex projective (or compact K\"ahler) manifolds with {\it nef tangent bundles}. Let $X$ be a complex projective manifold (or even a K\"ahler manifold) whose 
 tangent bundle is nef. Examples include rational homogeneous manifolds (e.g., complex projective spaces, flag manifolds, or quotients $G/P$ of a simply connected complex Lie group $G$ by a parabolic subgroup), abelian varieties (or complex tori), etc. (cf. \cite[Section 3.A]{DPS}). Moreover, as we will argue below, complex projective (or compact K\"ahler) manifolds with non-negative sectional curvature have nef tangent bundles.
By applying Theorem \ref{thm_nef} to the tangent bundle $TX$, one gets via the Gauss-Bonnet formula that 
\be\label{hp} \chi(X)=\int_X c_n(TX)=T_XX \cdot T_XX \geq 0,\ee
where $T_XX$ is the zero section of $TX$. Hence one gets the following result from \cite{DPS}.
\begin{prop}
\label{hn}
If $X$ is a complex projective (or compact K\"ahler) manifold 
with nef tangent bundle, then $\chi(X) \geq 0$. 
\end{prop}
Note that Theorem \ref{thm_nef} also shows that the inequality in \eqref{hp} is strict if $TX$ is ample. In fact, Mori \cite{Mor} proved that a complex projective manifold with ample tangent bundle is isomorphic to a complex projective space.

Proposition \ref{hn} is a generalization in the complex  projective (or K\"ahler) context of another conjecture of Hopf (also appearing on Yau's problem list \cite{Yau}):
\begin{conj}[Hopf]\label{ho2}
A compact, even-dimensional Riemannian manifold with positive sectional curvature has positive Euler characteristic. A compact, even-dimensional Riemannian manifold with non-negative sectional curvature has non-negative Euler characteristic.
\end{conj}
Indeed, let $X$ be a complex projective or compact K\"ahler manifold with non-negative (resp., positive) sectional curvature.  The bisectional curvature can be written as a positive linear combination of two sectional curvatures (e.g., see \cite{Z}), hence the bisectional curvature of $X$ is non-negative (resp., positive).  
This means (by definition) that the tangent bundle $TX$ is Griffiths semipositive (resp., positive). Furthermore, by \cite{DPS} (resp., \cite{Gri}), Griffiths semipositive (resp., positive) bundles are nef (resp., ample).

Let us finally note that in the K\"ahler context, Siu-Yau \cite{SY} showed that if $X$ has positive bisectional curvature then $X$ is biholomorphic to a complex projective space. Furthermore, a classification of all compact K\"ahler manifolds with  non-negative  bisectional curvature was obtained by Mok \cite{Mo}, and this can be used along with results of \cite{DPS} to give a direct proof of Conjecture \ref{ho2} in this case.


\end{document}